\documentclass[amscd,leqno,8pt]{article}

\usepackage[all]{xy}
\usepackage{amsthm}
\usepackage{amssymb} 
\usepackage{amsmath,amscd} 

\def\NN{{\Bbb N}}

\def\QQ{{\Bbb Q}}
\def\RR{{\Bbb R}}

\def\ZZ{{\Bbb Z}}

\def\11{{1\kern-3.5pt 1}}
\def\mumu{{\mu\kern-4.2pt\mu}}

\def\boxtimes{\setbox0\hbox{$\Box$}\copy0\kern-\wd0\hbox{$\times$}}

\DeclareMathOperator{\grotimes}{\underline{\otimes}}

\newcommand{\lotimes}{\otimes^{\bf{L}}}

\def\coker{\operatorname {coker}}

\def\Ext{\operatorname {Ext}}

\def\Hom{\operatorname {Hom}}
\def\id{\operatorname {id}}
\def\im{\operatorname {im}}

\def\ker{\operatorname {ker}}

\def\Tor{\operatorname {Tor}}

\def\Ext{\operatorname{Ext}}

\def\gldim{\operatorname{gldim}}

\def\Hom{\operatorname{Hom}}
\def\RHom{\operatorname{RHom}}

\def\mod{\operatorname{mod}}
\def\Mod{\operatorname{Mod}}

\def\Tor{\operatorname{Tor}}

\let\oldtext\text
\def\text#1{\oldtext{\normalshape #1}}

\def\L{\Lambda}

\def\Coh{\operatorname{Coh}}

\def\RHom{\textbf{R}\operatorname{Hom}}

\newtheorem{lemma}{Lemma}[section]
\newtheorem{proposition}[lemma]{Proposition}
\newtheorem{theorem}[lemma]{Theorem}
\newtheorem{corollary}[lemma]{Corollary}

{

}

\theoremstyle{definition}

\newtheorem{definition}[lemma]{\sl Definition}

\theoremstyle{remark}

\title{A criterion of graded coherentness of tensor algebras and its application to higher dimensional Auslander-Reiten theory.}
\author{Hiroyuki MINAMOTO}
\date{}

\begin{document}

\maketitle

\begin{abstract}
Even if a ring $A$ is coherent, 
the polynomial ring $A[X]$ in one variable could fail to be coherent. 
In this note we show that $A[X]$ is graded coherent with the grading $\deg X=1$. 
More generally, we give a criterion of graded coherentness of the tensor algebra $T_{A}(\sigma)$ 
of a certain class of  bi-module $\sigma$.

As an application of the criterion,  
we show that there is  a relationship between 
higher dimensional Auslander-Reiten theory and 
graded coherentness of higher preprojective algebras.
\end{abstract}

\section{Introduction}
In noncommutative projective geometry 
there is the important procedure  
which construct a graded ring
from certain data. 
The resultant graded ring is called the twisted homogeneous coordinate ring 
and plays an important role in noncommutative projective geometry and representation theory
(\cite{AZ,HIO,Min,MM,Pol}). 
This construction is similar 
to the procedure to take the tensor algebra $T_{A}(\sigma)$ over a ring $A$ of an $A-A$ bi-module $\sigma$ 
equipped with the grading $\deg \sigma := 1$.
In some case twisted homogeneous coordinate rings are actually tensor algebras. 
Although 
these are natural procedures to obtain graded rings,   
it is recognized that in general we can't expect that the resultant graded algebras are graded Noetherian 
and that only we can hope that these are graded coherent. 
For more details see \cite[Introduction]{Pol}, 
which suggest us that 
it is inevitable that 
we study graded coherentness more deeply.  

The condition of
(graded) coherentness is similar but weaker than that of 
(graded) Noetherianness. 
It is known that many theorems about finitely generated modules over (graded) Noetherian rings
 can be extended to finitely presented modules over (graded) coherent rings. 
However there is a significant difference between them. 
It is standard that 
if a ring $A$ is right  Noetherian, then the polynomial ring $A[X]$ in one variable is again right Noetherian. 
Contrary to this, 
it is well-known for specialists that 
even if a ring $A$ is right coherent,   
the polynomial ring $A[t]$ in one variable could fail to be coherent. 
In fact,   
Soublin \cite{Soublin} showed that the ring $A=\QQ[[x,y]]^{\Pi \NN}$ 
has such a property. 

In this note we show that if we take gradings into account, then the polynomial ring $A[X]$ become graded coherent 
for any coherent ring $A$. 
Namely we have    
\begin{theorem}[Corollary \ref{grcohcor}]
If a ring $A$ is right coherent, 
then the polynomial ring $A[X]$ with the grading $\deg X =1$ is graded right coherent. 
\end{theorem}
More generally 
we give 
a criterion of graded coherentness of the tensor algebra $T_{A}(\sigma)$ 
of a certain class of an $A-A$ bi-module $\sigma$. 
The condition on an $A-A$ bi-module $\sigma$ is of homological nature. 
Therefore the tensor algebra of a bi-modules satisfying the condition is well-behaved 
and contains an important example: higher preprojective algebra, which is studied in Section 3. 

In Section 3, we apply the criterion to show that 
there is a relationship between 
higher dimensional Auslander-Reiten theory of $n$-representation infinite algebra 
and 
graded coherentness of higher preprojective algebras.
(for the back ground and the definitions of these, please see the beginning of Section 3.)

\vspace{10pt}
\noindent
\textbf{Notations and conventions} 

Unless otherwise noted all modules and all graded modules considered are right modules 
and graded modules. 
For a graded module $M := \oplus_{n\in \ZZ} M_n$, we denote by $M(i)$ the $i$-degree shift, that is, $M(i)_{n}:= M_{n+i}$. 
We denote by $\grotimes$ the graded tensor product. 

\vspace{10pt}
\noindent
\textbf{Acknowledgment}
This work was supported by JSPS KAKENHI Grant Number 241065.

\section{A criterion of graded coherentness}

First we recall the definitions of coherent modules and coherent rings.

\begin{definition}
Let $A$ be a ring (resp. graded ring) and $M$ an $A$-module (resp. graded $A$-module). 

\begin{enumerate}
\item 
An $A$-module (resp. graded $A$-module) $M$ is called coherent (resp. graded coherent), 
if it satisfies the following two conditions: 

(i) $M$ is  finitely generated over $A$;  

(ii) 
for every $A$-homomorphism (resp. graded $A$-homomorphism) 
$f:P\to M$ with $P$ a finitely generated free $A$-module 
(resp. a finitely generated graded free $A$-module), 
the kernel $\ker(f)$ of 
is finitely generated over $A$. 

\item 
$A$ is called right coherent 
(resp. graded right coherent) 
if the regular $A$-module $A_{A}$ is coherent (resp. graded coherent).  

\end{enumerate}

\end{definition}
It is known that the full subcategory $\Coh A$ of coherent modules 
is an extension closed abelian subcategory of the category $\Mod A$ of $A$-modules.
A coherent $A$-module is finitely presented over $A$. 
If a ring  $A$ is right coherent. 
Then a $A$-module $M$ is coherent if and only if it is finitely presented over $A$. 
(See e.g. \cite{Pol}.)

\vspace{10pt}

Let $A$ be a ring.  
A complex $M^{\bullet}$ of $A$-modules is called \textit{pure} if its cohomology group concentrates in degree $0$,
i.e., $\textup{H}^{i}(M^{\bullet})=0 $ for $i \neq 0$.

\begin{theorem}\label{grcohcr}

Let $A$ be a right coherent  ring 
and $\sigma$  an $A$-$A$-bimodule 
such that $\sigma$ is finitely presented as right $A$-modules and 
the iterated derived tensor product  $\sigma^{\lotimes_{A} n}$ of $\sigma$ is pure for $n \geq 1$. 
We denote $\sigma^{\lotimes_{A} n}$ by $\sigma^n$. 
Let $T$ be the tensor algebra $T_{A}(\sigma)$ of $\sigma$ over $A$:
\[
T := A \oplus \sigma \oplus \sigma^2 \oplus \sigma^{3} \oplus \cdots 
\]
with the grading  $\deg \sigma:= 1$.

\begin{enumerate}
\item
Assume that 
for any finitely presented   $A$-module $M$ 
there is a natural number $m$ 
such that 
$(M \otimes_{A} \sigma^m) \lotimes_{A} \sigma^{n}$ 
is pure for $n \geq 0$. 
Then 
$T$ is graded right coherent. 

\item 
If $T$ has finite right graded global dimension, 
then the converse holds.  
Namely,  
if $T$ is graded right coherent, 
then 
for any finitely presented  $A$-module $M$ 
there is a natural number $m$ 
such that 
$(M \otimes_{A} \sigma^m) \lotimes_{A} \sigma^{n}$ 
is pure for $n \geq 0$. 
\end{enumerate}

\end{theorem}

Before giving a proof, we fix notations. 
For  graded $T$-module $X:= \oplus_{n\in \ZZ}X_{n}$,  
we denote by $\mu_{X,m,n}$ the $A$-homomorphism $X_{m}\otimes_{A}\sigma^{n} \to X_{m+n}$ 
induced by the multiplication $X \grotimes_{A} T \to X,\,\, x \otimes t \mapsto xt$.
Note that this $A$-homomorphism has the following compatibility with graded $T$-homomorphisms:
let $f:X \to Y$ be  a graded $T$-homomorphism. 
Then  
we have $f_{m+n} \circ \mu_{X,m,n} = \mu_{Y,m,n}\circ (f_{m} \otimes \id_{\sigma^{n}})$.
\[
\begin{xymatrix}
@C=20mm
{
X_{m} \otimes_{A} \sigma^{n} \ar[r]^{f_{m} \otimes \id_{\sigma^{n}}} \ar[d]_{\mu_{X,m,n}} 
&  Y_{m} \otimes_{A} \sigma^{n} \ar[d]^{\mu_{Y,m,n}} \\
 X_{m+n} \ar[r]_{f_{m+n}} & Y_{m+n} 
}\end{xymatrix}
\]

\begin{proof}
1. 
First note that 
we can easily check that $\sigma^n$ is coherent as a right $A$-module 
by induction on $n$. 

Let $P,Q$ be finitely generated graded free $T$-modules. 
We prove that 
the kernel $K:= \ker(f)$ of 
a graded $T$-homomorphism $f : P \to Q$ is finitely generated. 
Set $C:= \coker(f)$ and $I := \im(f)$.  
Then for any $s \in \ZZ$, we have $K_{s}= \ker(f_{s}),C_{s} = \coker(f_{s})$ and $I_{s} = \im(f_{s})$.  
In other words,  
looking at degree $s$-part we obtain the exact sequences of  $A$-modules: 
\[
\begin{split}
(\textbf{A}_s):  \quad  &0\to I_{s} \to Q_{s} \to C_{s} \to 0,\\
(\textbf{B}_s):  \quad &0\to K_{s} \to P_{s} \to I_{s}\to 0.
\end{split}
\]
Applying the functor $-\otimes_{A} \sigma^m$ to the exact sequence $(\textbf{A}_{s})$, 
we obtain 
the isomorphisms 
\[
(\textbf{C}_{s}): \quad \Tor^{A}_{i}(I_{s},\sigma^m) \cong \Tor^{A}_{i+1}(C_{s},\sigma^m) \textup{ for } i \geq 1 
\]
and 
the following commutative diagram $(\textbf{D}_{s})$:  
\[
\begin{xymatrix}{
0 \ar[r] & \Tor^{A}_{1}(C_{s},\sigma^m) \ar[r] & I_{s} \otimes_{A} \sigma^m \ar[d]_{\mu_{I,s,m}} \ar[r] & 
Q_{s} \otimes_{A} \sigma^m \ar[d]_{\mu_{Q,s,m}} \ar[r] & C_{s}\otimes_{A} \sigma^m\ar[d]_{\mu_{C,s,m}} \ar[r] & 0 \\
& 0 \ar[r] & I_{s+m} \ar[r] & Q_{s+m} \ar[r] & C_{s+m} \ar[r] &0. 
}\end{xymatrix}
\]
where the top and the bottom rows are exact.

We may assume that 
$P,Q$ are of form 
\[ P = \oplus_{p \leq i \leq q} T(-i)^{\oplus a_{i}}, Q = \oplus_{p\leq i \leq q} T(-i)^{\oplus b_{i}}. \] 
Then 
for $s \geq q$ 
we have 
$
P_{s} = 
\oplus_{p \leq i \leq q} (\sigma^{s- i})^{\oplus a_{i}},
Q_{s} =
 \oplus_{p \leq i \leq q} (\sigma^{s- i})^{\oplus b_{i}}$. 
Therefore 
for $s \geq q$ and $n \geq 0$
 the morphisms $\mu_{P,s,n}$ and $\mu_{Q,s,n}$ are isomorphisms. 
Hence by the right exactness of the functor $-\otimes_{A} \sigma^n$,  
 we see that the morphism 
$\mu_{C,s,n}: C_{s} \otimes_{A} \sigma^n \to C_{s+n}$ is an isomorphism for $s \geq q$ and $n \geq 0$.

By the first remark $P_{q}$ and $Q_{q}$ are coherent $A$-modules. 
Hence the cokernel $C_{q}$ of $f_{q} : P_{q} \to Q_{q}$ is a coherent $A$-module. 
Let $n$ be a natural number such that 
$C_{q+n} \lotimes_{A} \sigma^{m} \simeq (C_{q} \otimes_{A} \sigma^{n}) \lotimes_{A} \sigma^{m}$ is pure for $m \geq 0$. 
Then by the isomorphisms $(\textbf{C}_{q+n})$ we see that the complex $I_{q+n} \lotimes_{A} \sigma^{m}$ is pure. 
Moreover,  
since $\mu_{Q,q+n,m}$ and $\mu_{C,q+n,m}$ are isomorphisms 
and $\Tor_{1}^{A}(C_{q+n},\sigma^m)=0$,  
by  the commutative diagram $(\textbf{D}_{q+n})$, 
we see that the  morphism $\mu_{I,q+n,m}$ is an isomorphism. 

From the exact sequence $(\textbf{B}_{s})$,  
we obtain  the commutative diagram 
similar to  $(\textbf{D}_{s})$. 
Therefore,  
in the same way, 
we see that 
the morphism 
$\mu_{K,q+n,m}:  K_{q+n} \otimes_{A} \sigma^{m} \to K_{q+n+m}$ is an isomorphism  for $m\geq 0$.  
This shows that 
$K$ is generated by $\oplus_{p \leq i \leq q+n} K_{i}$ 
as a graded $T$-module. 
Since $K_{s}$ is the kernel of $A$-homomorphism between coherent modules, 
$K_{s}$ is coherent. In particular $K_{s}$ is finitely generated over $A$. 
Hence  
we conclude that $K$ is  finitely generated over $T$.

2. 
Let $M$ be a finitely presented $A$-module. 
Then we easily check that 
$M \grotimes_AT$ is a finitely presented graded $T$-module. 
Since $T$ is  graded right coherent and assumed to have finite right graded dimension, 
we have a graded  projective resolution of $M \grotimes_AT$ of finite length  
\[
(\textbf{E}): \quad 0 \to Q^{-g} \xrightarrow{d^{-g}} \cdots 
\xrightarrow{d^{-2}} Q^{-1} \xrightarrow{d^{-1}} Q^{0} \to M \grotimes_{A}T \to 0. 
\] 
such that 
each term $Q^{-i}$ is finitely generated graded $T$-module. 
Looking at degree $s$-part of $(\textbf{E})$ 
we obtain an exact sequence 
\[
(\textbf{E}_{s}): \quad 0 \to Q^{-g}_{s} \xrightarrow{d^{-g}_{s}} 
\cdots \xrightarrow{d^{-2}_{s}} Q^{-1}_{s} \xrightarrow{d^{-1}_{s}} Q^{0}_{s} \to M \otimes_{A}\sigma^{s} \to 0. 
\] 
By the same consideration as in the proof of (1), 
we see that 
there is a natural number $m$ 
such that the morphisms 
$\mu_{Q^{-i},m,n}:  Q^{-i}_{m} \otimes_{A} \sigma^n \to Q^{-i}_{m+n}$ are isomorphisms  
for $0 \leq i \leq g$ and $n \geq 0$.  
These isomorphisms induces an isomorphism 
$ (\textbf{E}_{m}) \otimes_{A} \sigma^n \xrightarrow{\cong} (\textbf{E}_{m+n})$ 
of  sequences of $A$-modules for $n \geq 0$. 
This shows that the sequence $(\textbf{E}_{m}) \otimes_{A} \sigma^{n}$ is exact.   
Since 
$(\textbf{E}_{m})$ is a $-\otimes_{A}\sigma^n$-acyclic resolution of $M \otimes_{A} \sigma^m$,  
 we conclude  that 
$(M \otimes_{A} \sigma^m) \lotimes_{A} \sigma^{n}$ 
is pure. 
\end{proof} 

In a case that we can a priori verify that $T_{A}(\sigma)$ is of finite right graded  global dimension, 
Theorem \ref{grcohcr} gives an equivalent condition that the tensor algebra $T_{A}(\sigma)$ 
to be graded right coherent. 
In particular we have 
\begin{corollary} 
Let $A$ and $\sigma$ be as in Theorem \ref{grcohcr}. 
Assume that $A$ is of finite right and left global dimension. 
Then the following conditions are equivalent: 
\begin{enumerate}
\item 
the tensor algebra $T_{A}(\sigma)$ with the grading $\deg \sigma = 1$ is graded right coherent, 

\item 
for any finitely presented   $A$-module $M$ 
there is a natural number $m$ 
such that 
$(M \otimes_{A} \sigma^m) \lotimes_{A} \sigma^{n}$ 
is pure for $n \geq 0$.
\end{enumerate}
\end{corollary}

\begin{proof}
By \cite[Theorem A.1]{Min} we have 
\[\textup{right graded gl.dim }  T_{A}(\sigma) \leq \textup{right gl.dim } A + \textup{left flat dim } {}_{A}\sigma + 1.\] 
Therefore by the assumption we see that $T_{A}(\sigma)$ is of finite right graded global dimension.
\end{proof}

Using the criterion of Theorem \ref{grcohcr}.1, 
we can easily check that the following graded algebras are graded coherent.  

\begin{corollary}\label{grcohcor}
If $A$ is right coherent,  
then the free algebra $A\langle X_{1}, \dots, X_{n} \rangle$ 
with the grading $\deg X_{i}=1$ for $i=1,\dots , n$ 
is graded right coherent.  
In particular the polynomial algebra $A[X]$ with the grading $\deg X = 1$ is graded right coherent. 
\end{corollary}

\begin{corollary}
Let $A$ be a finite dimensional algebra over a field $k$. 
Then the tensor algebra $T_{A}(A\otimes_kA)$ with the grading $\deg (A \otimes_{k} A) =1$ is graded coherent. 
\end{corollary}

Since $T_{A}(A \otimes_k A)$ is the bar resolution algebra, 
I hope that 
the above corollary has some application to Homological algebra of finite dimensional algebras. 


\section{Graded coherent-ness of $n+1$-preprojective algebra 
and its relationship to higher dimensional Auslander-Reiten Theory}

In \cite{Min} $n$-Fano algebras and its variants are introduced 
from a noncommutative algebro-geometric point of view.
Among these variants, 
$n$-quasi-extremely Fano algebra (qe Fano algebra for short) is defined by the simplest condition. 
In \cite{HIO} 
Herschend Iyama and Oppermann
called $n$-qe Fano algebra $n$-representation infinite algebra 
and showed that 
$n$-qe Fano algebra plays the role in higher dimensional Auslander-Reiten theory 
as representation infinite algebra in classical Auslander-Reiten theory of path algebras of quivers. 
Actually they introduced  the notion of $n$-hereditary algebra 
and proved  the dichotomy that 
an $n$-hereditary algebra is either  $n$-representation finite (which is defined in \cite{nRF}) 
or $n$-representation infinite. 
Moreover they studied the representation theory of $n$-representation infinite algebras and 
introduced the notions of $n$-preprojective, $n$-preinjective and $n$-regular modules for 
$n$-representation infinite algebras, which are shown to be 
 counterparts of preprojective, preinjective and regular modules 
 in the  representation theory of hereditary algebras of infinite representation type. 

The Auslander correspondence tells us that 
the representation theory of a representation finite algebra $A$ is equivalent to 
the ring theory of the Auslander algebra $\textup{Aus}(A)$ of $A$. 
In higher dimensional Auslander-Reiten theory for $n$-representation infinite algebras, 
the object which is expected to be the counterpart of the Auslander algebra 
is the $n+1$-preprojective algebra $\Pi_{n+1}(\L)$. 
Therefore it is expected that 
the representation theory of $\L$ has close relationship to the ring theory of $\Pi_{n+1}(\L)$. 
We will observe such a phenomenon in Proposition \ref{ppcoh}, 
which suggest  that 
graded coherentness of $n+1$-preprojective algebra has an importance in higher dimensional Auslander-Reiten theory.

We recall the 
definition of $n$-representation infinite algebras (or $n$-qe Fano algebras) 
and the $n+1$ preprojective algebra $\Pi_{n+1}(\L)$.
Let $k$ be a field.

\begin{definition}
Let $n$ be a natural number. 
A finite dimensional $k$-algebra $\L$ is called an $n$-representation infinite algebras  
if the following conditions are satisfied: 
\begin{enumerate}
\item the algebra $\L$ is of finite global dimension, 

\item 
Let $D(\L):= \Hom_{k}(\L,k)$ be the $k$-dual of $\L$ equipped with the natural bi-module structure. 
We define a complex $\theta$ of $\L-\L$ bi-modules to be  $\theta:= \RHom(D(\L),\L)[n]$. 
Then the complex $\theta^{\lotimes_{\L} s}$ is pure for $s \geq 1$. 

\end{enumerate}
(Note that by the second condition for $s=1$, 
we may consider the complex $\theta$ as the module $\Ext_{\L}^{n}(D(\L),\L)$.) 

We denote $\theta^{\lotimes_{\L} s}$ by $\theta^s$. 
The $n+1$-preprojective algebra $\Pi_{n+1}(\L)$  
is defined to be the tensor algebra $T_{\L}(\theta)$. 
\[
\Pi_{n+1}(\L) = \L \oplus \theta \oplus \theta^2 \oplus \theta^3 \oplus \cdots .
\]
We equip $\Pi_{n+1}(\L)$ with the grading $\deg \theta := 1$.
\end{definition}

We define the $n$-Auslander-Reiten translations by 
\[
\tau_{n} :=- \otimes_{\L}\theta \textup{ and } \tau^{-}_{n}:=\Hom_{\L}(\theta,-): \mod \L \to \mod \L.
\] 
Note that the functors $\tau_{n}$ and $\tau_{n}^{-}$ are \textit{not} quasi-inverse to each other 
but only an adjoint pair. 
In other words the unite morphism  
\[
\eta_{M}:M \to \tau_{n}^{-}\tau_{n}M
\] 
of the adjunction $\tau_{n} 
\dashv\tau_{n}^{-}$  is not an isomorphisms for general $M \in \mod \L$.

For a $\L$-module $M$ and a natural number $s$, we denote  by $\eta_{M,s}$ the $\L$-homomorphism  
\[
\eta_{M,s}: \tau_{n}^{s}\tau_{n}^{-s}M
\to \tau_{n}^{s+1}\tau_{n}^{-(s+1)}M  
\]
which is induced by the adjunction $\tau_{n} \dashv \tau_{n}^{-}$.  
We can easily check that if $M$ is either $n$-preprojective, $n$-preinjective or $n$-regular, 
then $\eta_{M,s}$ is an isomorphism for $s \gg 0$.
For a general $\L$-module $M$, we don't know that $\eta_{M,s}$ become an isomorphism for $s \gg 0$. 
The following proposition tells us 
that 
this problem relates to the graded coherentness of the $n+1$-preprojective algebra. 

\begin{proposition}\label{ppcoh}
We retain the above  notations.
\begin{enumerate} 
\item
If $\Pi_{n+1}(\L)$ is right graded coherent, then for each finitely generated $\L$-module $M$ 
there is $s_{0}$ such that $\eta_{M,s}$ is an isomorphism  for $s \geq s_{0}$. 
 
\item 
In the case when $n\leq 2$, 
then the converse holds. 
Namely 
if 
for each finitely generated $\L$-module $M$ 
there is $s_{0}$ such that $\eta_{M,s}$ is an isomorphism  for $s \geq s_{0}$, 
then $\Pi_{n}(\L)$ is graded coherent. 
\end{enumerate}
\end{proposition}


\begin{proof}
(1) 
By \cite[Theorem 4.2]{MM} 
the right graded global dimension of $\Pi_{n+1}(\L)$ is $n+1$. 
In particular it is finite.   
Now the verification  is an easy application of Theorem \ref{grcohcr}.2.

(2) 
First note that we have an isomorphism  
$\tau_{n}^{-s}\tau_{n}^sM \cong \Hom_{\L}(\theta^s,M\otimes_{\L}\theta^s)$. 
For $t \geq 0$
we have 
\[
(M \otimes_{\L} \theta^{s_0}) \lotimes_{\L} \theta^{t}
\simeq \Hom(\theta^{s_{0}}, M \otimes_{\L} \theta^{s_{0}})\lotimes_{\L} \theta^{s_{0}+t} 
\simeq \Hom(\theta^{s_{0}+t}, M \otimes_{\L} \theta^{s_{0}+t})\lotimes_{\L} \theta^{s_{0}+t} \simeq M \otimes_{\L}\theta^{s_{0}+t}
\]
where 
the first quasi-isomorphism follows from the following Lemma \ref{2dim},  
the second quasi-isomorphism follows from the assumption and the third quasi-isomorphism follows from  Lemma \ref{2dim}. 
Hence we see that $(M \otimes_{\L} \theta^{s_0}) \lotimes_{\L} \theta^{t}$ is pure for $t\geq 0$.  
By Theorem \ref{grcohcr}.1  we verify the claim. 
\end{proof}

\begin{lemma}\label{2dim}
Let $A$ be a  ring of $\gldim A \leq 2$.
Let $\sigma$ be an $A$-$A$-bi-module such that 
the complex 
$\RR\Hom(\sigma,\sigma)$ is pure. 
Then 
for a finitely presented   $A$-module $M$, 
the complex 
$\RR\Hom(\sigma,M \otimes_{A} \sigma)$ is pure. 
\end{lemma}

\begin{proof}
Take a finite presentation $P \to Q \to M \to 0$ and compute $\Ext^{1,2}_{A}(\sigma,M \otimes_{A} \sigma)$.  
Standard homological algebra shows the claim. 
\end{proof}

\end{document}